\documentclass[11pt, usenames]{article}
\usepackage{geometry}                
\usepackage{graphicx}
\usepackage{amssymb}
\usepackage{epstopdf}
\usepackage[T1]{fontenc}
\usepackage[ngerman, english]{babel}
\usepackage{mathtools}
\usepackage{bbm}
\usepackage{mathrsfs}
\usepackage{amsmath,amscd}
\usepackage{tikz}
\usepackage{tikz-cd}
\usepackage{babel}

\usetikzlibrary{arrows, matrix}
\usetikzlibrary{cd}
\usepackage[arrow, matrix, curve]{xy}

\newcommand{\qed}{\hfill \ensuremath{\Box}}
\newenvironment{proof}{\vspace{1ex}\noindent{\it Proof.}\hspace{0.5em}}
	{\hfill\qed\vspace{1ex}}

\newtheorem{theorem}{Theorem}[section]
\newtheorem{lemma}[theorem]{Lemma}
\newtheorem{proposition}[theorem]{Proposition}

\newtheorem{definition}[theorem]{Definition}

\DeclareMathOperator{\Gal}{\operatorname{Gal}}
\DeclareMathOperator{\Q}{\mathbf{Q}}
\DeclareMathOperator{\R}{\mathrm{R}}
\DeclareMathOperator{\Z}{\mathbf{Z}}
\DeclareMathOperator{\F}{\mathbf{F}} 

\DeclareMathOperator{\A}{\mathbf{A}}
\DeclareMathOperator{\Ps}{\mathbf{P}}

\DeclareMathOperator{\N}{\mathbf{N}}

\DeclareMathOperator{\im}{\mathrm{im}}
\DeclareMathOperator{\Spec}{\operatorname{Spec}}

\DeclareMathOperator{\Hom}{\operatorname{Hom}}
\DeclareMathOperator{\Frac}{\operatorname{Frac}}

\DeclareMathOperator{\ord}{\operatorname{ord}}

\DeclareMathOperator{\Og}{\mathcal{O}}
\DeclareMathOperator{\pr}{\mathrm{pr}}

\DeclareMathOperator{\Pic}{\mathrm{Pic}}

\DeclareMathOperator{\Gm}{\mathbf{G}_m}

\DeclareMathOperator{\alg}{{^\mathrm{alg}}}
\DeclareMathOperator{\sep}{{^\mathrm{sep}}}
\DeclareMathOperator{\perf}{{^\mathrm{perf}}}
\DeclareMathOperator{\gp}{{^\mathrm{gp}}}
\title{On logarithmic reduction of cohomologically tame\\ elliptic surfaces}
\author{Otto Overkamp and Arne Smeets}
\date{}

\begin{document}
\maketitle
{\abstract{We give cohomological criteria for logarithmic good reduction of elliptic surfaces up to modification. Along the way, we prove several more general results about such surfaces in positive characteristic, as well as about log smooth morphisms.}}
\tableofcontents

\section{Introduction}
Let $\Og_K$ be a complete discrete valuation ring with residue field $\kappa$ and field of fractions $K.$ Throughout, we shall assume that $\mathrm{char}\, \kappa\not\in \{2,3\},$ and that $\kappa$ is algebraically closed. Choose an algebraic closure $K\alg$ of $K$ and let $K\sep$ be the separable closure of $K$ in $K\alg$. In this paper, we shall be interested in studying the reduction of \it elliptic surfaces \rm in the realm of logarithmic geometry. Throughout this paper, we shall use the following
\begin{definition}
Let $k$ be an arbitrary field. An \rm elliptic surface \it over $k$ is a morphism $f\colon X\to C,$ where $X$ is a smooth, projective, and geometrically integral algebraic surface over $k$ and $C$ is a smooth, projective and geometrically integral curve over $k$ with generic point $\eta$ with the following two properties:\\
(i) The generic fibre $X_\eta$ of $f$ is smooth over $\kappa(\eta)$ with $\dim_{\kappa(\eta)} H^1(X_{\eta}, \Og_{X_{\eta}})=1$, and\\
(ii) the morphism $f$ admits a section $\sigma\colon C\to X.$    
\end{definition}
This definition has several immediate consequences. For example, the morphism $f$ is automatically flat, and $X_{\eta}$ is, in fact, an \it elliptic curve \rm over $\kappa(\eta).$ By passing to the limit, one can furthermore show that there exists a non-empty open subset $U\subseteq C$ such that the base change of $f$ to $U$ is smooth. Note that many authors impose additional conditions on the morphism $f$ when defining elliptic surfaces. For example, it is often required that the fibration be \it minimal \rm (see Definition \ref{mindef}), or that $f$ possess at least one fibre which is not smooth over the residue field. We shall not impose any such restrictions; in particular, surfaces of the form $E\times_{\Spec k} C$, where $E$ is an elliptic curve over $k$ and $C$ a curve as in the Definition, will count as elliptic surfaces over $k,$ and our results apply to such surfaces. However, many results (such as our main theorem) will be formulated for minimal elliptic surfaces; this will always be stated explicitly. \\
Now let $f\colon X\to C$ be a minimal elliptic surface over the discretely valued field $K$ and let $\ell$ be a prime number invertible in $\Og_K.$ Then, for $i\geq 0,$ the finitely generated $\Z_\ell$-modules $H^i(X_{K\sep}, \Z_\ell)$ are representations of the absolute Galois group $\Gal(K\sep/K)$ of $K$ in a natural way (all cohomology in this paper will be \'etale cohomology unless stated otherwise). This Galois group contains the subgroup $P,$ which is commonly called the \it wild inertia subgroup, \rm defined as $P:=\Gal(K\sep/K^{\mathrm{t}}),$ where $K^{\mathrm{t}}$ is the maximal tamely ramified subextension of $K\subseteq K\sep.$ If $M$ is an Abelian group on which $\Gal(K\sep/K)$ acts, we shall call the representation $M$ \it tamely ramified \rm if the induced representation of $P$ on $M$ is trivial. Intuitively speaking, it is expected that tamely ramified representations should, when studying the reduction of algebraic varieties over $K$ in the realm of logarithmic geometry, play a similar role to that played by unramified representations in the classical case. The results of the present paper can be seen as confirming this philosophy in the case of elliptic surfaces. Our main result is
\begin{theorem} Let $f\colon X\to C$ be a minimal elliptic surface over $K.$ Suppose that one of \rm Condition 1 \it or \rm Condition 2 \it (see p. 14 and p. 15, respectively) are satisfied.  Let $D$ be the formal sum of closed points $x$ of $C$ such that the fibre of $f$ above $x$ is not smooth. Assume that there exists an effective divisor $A$ on $C$ which is \'etale and tamely ramified over $K$ such that 
$$\deg_K A+\deg_KD+2\dim_KH^1(C, \Og_C)>2,$$ and such that the supports of $D$ and $A$ do not intersect. Then there exists a regular scheme $\mathscr{X}$ together with a flat and projective morphism $f\colon \mathscr{X}\to \Spec \Og_K$ with the following properties: \\
(i) The morphism $\mathscr{X}\to \Spec \Og_K$ is log smooth with respect to the log structure given by the special fibres, and\\
(ii) There exists a projective modification $\mu\colon \mathscr{X}\times_{\Spec\Og_K}\Spec K \to X.$
Furthermore, the morphism $f$ maps the exceptional locus of $\mu$ into the support of $D.$ 
\end{theorem}
See Theorem \ref{mainthm} for a more precise statement. Condition 1 is satisfied, for example, if all singular fibres of $f$ have at least two geometric irreducible components and the Galois representations $H^i(X_{K\sep}, \Q_\ell)$ are tamely ramified. The numerical condition involving $D,$ $A,$ and $H^1(C, \Og_C)$ is only of interest if the genus of $C$ is equal to 0 or 1, as the condition is always satisfied otherwise even if $D=A=0.$ If the genus is equal to 1, the condition excludes base curves which do not have points in tamely ramified extensions of $K;$ for such curves, the question of logarithmic good reduction is much more delicate than otherwise. If the genus of $C$ is zero (which implies $C\cong \Ps^1_K$ in our situation), the condition is always satisfied as well. However, if $\deg D\in \{1,2\},$ one would have to choose $A$ inside the open subset of $C$ above which $f$ is smooth (see Theorem \ref{smooththm} for the case $\deg D=0$). The relationship of cohomological tameness (i.e., tame ramification of the Galois representations $H^i(X_{K\sep}, \Z_\ell)$) and logarithmic good reduction has been investigated before in various situations; if $X$ is a curve, see \cite{Stix}; if $X$ is an Abelian variety, see \cite{BS}. \\
\\
$\mathbf{Acknowledgement}.$ This paper was written while the first author was a postdoctoral researcher at Leibniz Universität Hannover and revised while he was working at the Mathematical Institute of the University of Oxford, supported by a research fellowship from the  German Research Foundation (Deutsche Forschungsgemeinschaft; Geschäftszeichen OV 163/1-1, Projektnummer 442615504). He would like to express his gratitude to all three of those institutions for their support. The second author acknowledges support of an NWO Veni grant (639.031.756) and a KU Leuven start-up grant.

\section{Elliptic surfaces over separably closed fields}

Let $k$ be a separably closed field of characteristic $p\not\in \{2,3\},$ and denote by $k{\alg}$ an algebraic closure of $k.$ Suppose
$f\colon X\to C$
is an elliptic surface over $k$. By assumption, $f$ is flat and smooth away from finitely many points on $C$. Assume that the fibration $X\to C$ admits a section $\sigma\colon C\to X.$
Let $x\in C$ be a closed point. Since $C_{k\alg}:=C\times_{\Spec k}\Spec k\alg\to C$ is a homeomorphism, there is a unique closed point $\overline{x}$ of $C_{k\alg}$ lying above $x.$ The corresponding local rings $\Og_{C,x}$ and $\Og_{C_{k\alg}, \overline{x}}$ are excellent discrete valuation rings satisfying $$\Og_{C,x}\otimes_kk\alg=\Og_{C_{k\alg}, \overline{x}}.$$ 
\begin{lemma}
Let $\kappa(x)$ denote the the residue field of $x$. The local extension of discrete valuation rings $\Og_{C,x}\subseteq \Og_{C_{k\alg},\overline{x}}$ has ramification index equal to $[\kappa(x):k]$, i.e., the normalised valuation of a uniformiser of $\Og_{C,x}$ in $\Og_{C_{k\alg},\overline{x}}$ is equal to $[\kappa(x):k]$. \label{Oramlem}
\end{lemma}
\begin{proof}
Let $y\in C_{\kappa(x)}:=C\times_{\Spec k}\Spec \kappa(x)$ be the unique closed point of $C_{\kappa(x)}$ lying above $x.$ Then we can factor the extension of discrete valuation rings as
$\Og_{C,x}\subseteq \Og_{C_{\kappa(x)},y}\subseteq \Og_{C_{k\alg}, \overline{x}}.$ We claim that the index of ramification of the first extension is $[\kappa(x):k],$ whereas the index of the second one is $1$. To prove the first claim, note that we have 
$$\Og_{C_{\kappa(x)},y}=\Og_{C,x}\otimes_k\kappa(x).$$
In particular, the first extension is finite, of degree $[\kappa(x):k].$ On the other hand, the residue field extension of $\Og_{C,x}\subseteq \Og_{C_{\kappa(x)},y}$ is trivial. Hence we obtain the first claim. For the second one, let $\xi$ be a function defined in a Zariski neighbourhood of $X$ which cuts out $y$ scheme-theoretically. Then the image of $\xi$ in $\Og_{C_{\kappa(x)},y}$ is a uniformiser. The subscheme cut out by the pullback of $\xi$ to $C_{k\alg}$ is isomorphic to $\Spec\kappa(x)\times_{\Spec \kappa(x)}\Spec k\alg,$ which is reduced. Hence this pullback of $\xi$ is a uniformiser of $\Og_{C_{k\alg},\overline{x}}$ and the claim follows. 
\end{proof}\\
\subsection{Minimal proper regular models}
The goal of this subsection is to show that the elliptic surface $\overline{f}\colon X_{k\alg}\to C_{k\alg}$ is minimal as well. Let us first recall the following definition (adapted from \cite[Definition 3.8]{Con}): 
\begin{definition}
Let $S$ be a Dedekind scheme with generic point $\eta$, and let $\mathscr{E}\to S$ be a flat, proper morphism such that $\mathscr{E}$ is regular and such that the generic fibre $\mathscr{E}_{\eta}$ is an elliptic curve. Then $\mathscr{E}$ is a minimal proper regular model of $\mathscr{E}_{\eta}$ if and only if the following holds: if $\mathscr{E}\to \mathscr{E'}$ is a morphism from $\mathscr{E}$ to another regular proper model $\mathscr{E}'$ of $\mathscr{E}_{\eta}$ which becomes an isomorphism on the generic fibre, then $\mathscr{E}\to \mathscr{E'}$ is a global isomorphism. \label{mindef}
\end{definition}
If $E$ is any elliptic curve over the field of fractions of $S$, then there exists a \it unique \rm minimal proper regular model of $E$ over $S$ (see, e.g., \cite[Theorem 1.2]{Chi}). The minimal proper regular model $\mathscr{E}$ of a given elliptic curve $E$ satisfies the following universal property: For every proper regular model $\mathscr{E}''$ of $E$ over $R$, there exists a unique morphism $\mathscr{E}''\to \mathscr{E}$ respecting the identification of the generic fibres (see \cite[Remark 3.11]{Con}).

Let $\mathscr{E}\to S$ be a proper, flat morphism $\mathscr{E}\to S$, with $\mathscr{E}$ regular, with generic fibre an elliptic curve $E$. Then we have the following criterion which tells us whether $\mathscr{E}$ is minimal: 

\begin{proposition} \rm (Castelnuovo's criterion) \it
Let $S$ be a Dedekind scheme with generic point $\eta,$ and let $s$ be a closed point of $S.$ Let $\mathscr{E}\to S$ be a proper and flat morphism such that $\mathscr{E}$ is regular and such that $\mathscr{E}_\eta$ is an elliptic curve over $\Spec \kappa(\eta).$ Then the following are equivalent: \\
(i) the scheme $\mathscr{E}$ is a minimal proper regular model of its generic fibre, and\\
(ii) for any closed point $s\in S,$ the scheme $\mathscr{E}\times_S \Spec \Og_{S,s}$ is a minimal proper regular model of its generic fibre, and \\
(iii) for any closed point $s\in S,$ there does not exist an irreducible component $Z$ of $\mathscr{E}_s$ (endowed with its reduced closed subscheme structure) with the following property: Let $k_Z:=\Gamma(Z, \Og_Z);$ this is a finite field extension of $k.$ Denote by $\iota\colon Z\to \mathscr{E}$ the canonical closed immersion. Then 
$$\deg_{k_Z} \iota^\ast \Og_{\mathscr{E}}(Z)=-1$$ and $Z\cong \Ps^1_{k_Z}.$ \label{Castelnuovoprop}
\end{proposition}
\begin{proof}
This is an immediate consequence of \cite{Chi}, Theorems 1.2, 2.1, and  3.1.
\end{proof}\\
Assume that $f\colon X\to C$ is a minimal elliptic surface over $k.$ We will now prove that the irreducible components of a singular fibre are projective lines over the appropriate residue fields, provided that the fibre has at least two irreducible components. 

\begin{lemma}
Let $f\colon X\to C$ be a minimal elliptic surface over $k.$ Let $x\in C$ be a closed point such that the fibre above $x$ has at least two irreducible components. Then each component is isomorphic to $\Ps^1_{\kappa(x)},$ and $\kappa(x)=k.$ \label{P1lem}
\end{lemma}
\begin{proof} We begin by showing that $\Ps^1_{k}$ can be characterised as follows: Let $Z$ be a proper integral curve over $k$ such that $\Gamma(Z, \Og_Z)=k$ and $H^1(Z, \Og_Z)=0$, and assume that $Z$ admits a line bundle of $k$-degree not divisible by $p=\mathrm{char}\, k.$ Then $Z\cong \Ps^1_{k}.$ 

Indeed, we see straight away that $Z$ is regular: If it were not, then then its normalisation would be isomorphic to $\Ps^1_{k'}$ for some non-trivial finite extension $k'$ of $k$ \cite[Tag 0DJB]{Stacks}. In particular, all line bundles on $Z$ would have degree divisible by $[k':k]$; since $k$ is separably closed, this is impossible. Therefore $Z$ is regular, and $Z$ is isomorphic to a plane projective curve of degree $2$ over $k$ (see \cite[Tag 0C6N]{Stacks}). Moreover $Z$ has a $k$-rational point since $\mathrm{char}\ k >2 $, which implies that $Z\cong \Ps^1_{k}$, as desired. 

Let us now assume that $Z$ is an irreducible component of the fibre of $f$ above $x$, endowed with its reduced subscheme structure. To see that $\Gamma(Z, \Og_Z)=\kappa(x)=k$, let $k_Z:=\Gamma(Z, \Og_Z)$, which is a finite extension of $\kappa(x)$, and consider the commutative diagram
$$\begin{tikzcd}
&X_{k_Z} \arrow{d}{\pi}\\
Z\arrow[swap]{r}{\iota}\arrow{ru}{\iota_Z}&X.
\end{tikzcd}$$
Note that $X_{k_Z}$ is smooth over $k_Z.$ The same argument as in the proof of Lemma \ref{Oramlem} shows that $\pi^\ast \Og_{X}(Z)=\Og_{X_{k_Z}}(Z)^{\otimes{[k_Z:k]}},$ which implies
\begin{align*}
[k_Z:k]\deg_{k_Z}\iota_Z^\ast \Og_{X_{k_Z}}(Z)=\deg_{k_Z} \iota^\ast \Og_X(Z).
\end{align*}
It follows from the adjunction formula 
$$ \deg_{k_Z} \iota^\ast \Og_{X}(Z)=-2+ 2\dim_{k_Z} H^1(Z, \Og_Z)$$ (see \cite[pp. 20f]{Con}; here we use minimality), together with the fact that $\deg_{k_Z} \iota^\ast \Og_{X}(Z) < 0$, that $\dim_{k_Z} H^1(Z, \Og_Z)=0$ and $\deg_{k_Z} \iota^\ast \Og_{X}(Z)=-2.$ We conclude that $k_Z=k$ and therefore also $k_Z=\kappa(x)$, which finishes the proof.
\end{proof}\\
Like the previous results in the section, the following statement is probably well-known to experts, but we were not able to find a reference; the proof is not complicated.
\begin{proposition}
Let $f\colon X\to C$ be a minimal elliptic surface over $k$. Then the induced fibration $\overline{f}\colon X_{k\alg}\to C_{k\alg}$ is minimal as well. \label{generalminprop}
\end{proposition}
\begin{proof}
Let $x$ be a closed point of $C$ above which the fibre of $X$ is singular, and let $\overline{x}$ be the unique preimage of $x$ in $C_{k\alg}.$ We will show that no irreducible component of $\overline{f}^{-1}(\overline{x})$ satisfies the conditions from Castelnuovo's criterion. If $\overline{Z}$ were such a component, then since $f^{-1}(x)$ and $\overline{f}^{-1}(\overline{x})$ are homeomorphic, there would be a unique  component $Z$ of $f^{-1}(x)$ such that $\overline{Z}=(Z\times_{\Spec\kappa(x)}\Spec k\alg)_{\mathrm{red}}.$ It is clear that $\overline{Z}$ cannot be the only component of $\overline{f}^{-1}(\overline{x})$. Therefore Lemma \ref{P1lem} tells us that $Z$ is geometrically reduced over $\kappa(x)=k,$ so that we have $\overline{Z}=Z\times_{\Spec k}\Spec k\alg$ scheme-theoretically.
We now consider the diagram
$$\begin{tikzcd}
\overline{Z} \arrow{r}{\overline{\iota}} \arrow[swap]{d}{\pr_Z}&X_{k\alg}\arrow{d}{\pr}\\
Z\arrow[swap]{r}{\iota}&X.
\end{tikzcd}$$ 
Then we have $\pr^\ast \Og_X(Z)=\Og_{X_{k\alg}}(\overline{Z}),$ and hence
$$\deg_{k\alg} \overline{\iota}^\ast \Og_{X_{k\alg}}(\overline{Z})=\deg_k\iota^\ast \Og_X(Z)=-2,$$ contradicting our initial assumption that $\deg_{k\alg} \overline{\iota}^\ast \Og_{X_{k\alg}}(\overline{Z}) = -1$. 
\end{proof}
\begin{proposition}
Let $k$ be an arbitrary field of characteristic not equal to $2$ or $3$. Let $f\colon X\to C$ be a minimal elliptic surface over $k$. Then the induced morphisms $X_{k\alg}\to C_{k\alg}$ and $X_{k\sep}\to C_{k\sep}$ are also minimal.  \label{generalminpropII}
\end{proposition}
\begin{proof}
By Proposition \ref{generalminprop}, all we have to show is that $X_{k\sep}\to C_{k\sep}$ is minimal. Let $x\in C$ be a closed point and let $y\in C_{k\sep}$ be a preimage of $x$ under the canonical projection. The local extension $\Og_{C, x}\subseteq \Og_{C_{k\sep}, y}$ has ramification index $1$. Since $X$ is smooth over $k$, we know that $X_{k\sep}$ is smooth over $k\sep$, and therefore $X_{k\sep}\times_{C_{k\sep}} \Spec \Og_{C_{k\sep},y}$ is a proper regular model of its generic fibre. A minimal Weierstrass equation of $X\times_C\Spec \Og_{C,x}$ also yields a Weierstrass equation of $X_{k\sep}\times_{C_{k\sep}} \Spec \Og_{C_{k\sep},y}$ over $\Og_{C_{k\sep},y}$. Tate's algorithm (see \cite[pp. 364ff]{Silv} and \cite[Lemma 1]{Szy}) shows that the reduction type of the special fibre of the minimal proper regular model $\mathscr{E}$ of the generic fibre of $X_{k\sep}\times_{C_{k\sep}} \Spec \Og_{C_{k\sep},y}$ is the same as the reduction type of $X\times_C \Spec \Og_{C, x}.$ This implies that the canonical morphism $X_{k\sep}\times_{C_{k\sep}} \Spec \Og_{C_{k\sep},y}\to \mathscr{E}$ is an isomorphism, which proves our claim.
\end{proof}
\subsection{Auxiliary results on elliptic surfaces over arbitrary fields}
We now collect some auxiliary results on elliptic surfaces $f\colon X\to C$ over general base fields $k$, which we do not assume to be perfect or separably closed. These results show that minimal elliptic surfaces are quite ``common'', in the sense that any elliptic surface can be turned into a minimal one without losing smoothness, and that any K3 surface which admits the structure of an elliptic surface is automatically minimal.  Of course such results are well-documented over more restrictive classes of base fields.
\begin{proposition}
Let $f\colon X\to C$ a (not necessarily minimal) elliptic surface over $k$, and let $f' \colon X' \to C$ be the model of the generic fibre of $f$ obtained by contracting an exceptional curve on $X$ contained in a fibre of $f$. Then $X'$ is smooth over $k$. In particular, the minimal proper regular model of $f$ is smooth over $k.$ 
\end{proposition}
\begin{proof}
Let $Z$ be a an exceptional curve, i.e., a reduced irreducible component of a fibre of $f$ such that, if $k_Z:=\Gamma(Z, \Og_Z)$ and $\iota\colon Z\to X$ denotes the canonical closed immersion, we have $Z\cong \Ps^1_{k_Z}$ and $\deg_{k_Z}\iota^\ast \Og_X(Z)=-1$. Let $\ell$ be the separable closure of $k$ in $k_Z.$ Consider the commutative diagram
$$\begin{tikzcd}  \label{CD}
& X_{k_Z} \arrow{d}{\pi}\\
Z \arrow{ru}{\iota_{k_Z}} \arrow{r}[swap]{\iota_\ell}\arrow{rd}[swap]{\iota} & X_\ell \arrow{d}{\pi'}\\
&X,
\end{tikzcd}$$
where $X_\ell:=X\times_{\Spec k}\Spec \ell$ and similarly for $k_Z.$ The map $Z\to Z\times_{\Spec k}\Spec \ell$ identifies $Z$ (scheme-theoretically) with a connected component of $Z\times_{\Spec k}\Spec \ell.$ Hence we can find an effective divisor $D$ on $X_\ell$ such that $\pi'^\ast\Og_X(Z)\cong \Og_{X_\ell}(Z+D)$ and such that the support of $D$ does not intersect that of $Z.$ In particular, we have $\iota_\ell^\ast\pi'^\ast\Og_X(Z)\cong \iota_\ell^\ast\Og_{X_\ell}(Z),$ whence
$$
-1=\deg_{k_Z}\iota^\ast \Og_X(Z) =\deg_{k_Z} \iota_\ell^\ast \Og_{X_\ell}(Z).
$$
We claim that, in fact, $k_Z=\ell.$ Indeed, we have
$$
-1=\deg_{k_Z} \iota_{\ell}^\ast \Og_{X_{\ell}}(Z) =\deg_{k_Z} \iota_{k_Z}^\ast \pi^\ast \Og_{X_\ell}(Z) =[k_Z:\ell] \deg_{k_Z} \iota_{k_Z}^\ast \Og_{X_{k_Z}}(Z).$$
The third equality can be seen as follows. Recall that $\pi$ is a homeomorphism, let $\mathfrak{z}$ be a generic point of $Z$ in $X_{\ell}$, and let $\mathfrak{z}'$ be its unique preimage in $X_{k_Z}$. Then $\Og_{X_{\ell}, \mathfrak{z}} \subseteq \Og_{X_{k_Z}, \mathfrak{z}'}$ is a finite extension of discrete valuation rings with  ramification index equal to $[k_Z:\ell].$ This implies that $\pi^\ast\Og_{X_{\ell}}(Z)=\Og_{X_{k_Z}}(Z)^{\otimes[k_Z:\ell]}.$ Therefore we obtain that $[k_Z:\ell]=1.$

The assumptions imply that we can contract $Z:$ We obtain a regular scheme $X'\to C$ together with a contraction map $X\to X'$. It remains to prove that $X'$ is smooth over $k.$ 

All of $Z$ is mapped to a single point $q\in X'$ with the property that $\kappa(q)=\ell$. Let $\tilde{\ell}$ be a Galois closure of $\ell$ over $k$. It suffices to show that $X'\times_{\Spec k}\Spec \tilde{\ell}$ is smooth over $\tilde{\ell}$. Let $\tilde{\ell}\perf$ denote the perfect closure of $\tilde{\ell}.$ Then $\Spec \tilde{\ell}\perf\to \Spec \tilde{\ell}$ is a universal homeomorphism, and we want to show that $X'\times_{\Spec k} \Spec \tilde{\ell}\perf$ is regular. This is clear at all points of $X'\times_{\Spec k} \Spec \tilde{\ell}\perf$ which do \emph{not} lie over the image of $\Spec \ell\times_{\Spec k}\Spec \tilde{\ell}\to X'\times_{\Spec k} \Spec \tilde{\ell};$ this image consists of finitely many $\tilde{\ell}$-rational points of $X\times_{\Spec k} \Spec \tilde{\ell}.$ 

Let $x$ be a closed point in this image, with unique preimage $y$ in the fibre product $X'\times_{\Spec k}\Spec\tilde{\ell}\perf.$ Denoting $X'_{\tilde{\ell}}:=X'\times_{\Spec k}\Spec \tilde{\ell}$ (and similarly for $\tilde{\ell}\perf$), we have $$\Og_{X'_{\tilde{\ell}\perf}, y}=\Og_{X'_{\tilde{\ell}}, x}\otimes_{\tilde{\ell}}\,\tilde{\ell}\perf.$$ Since $\kappa(x)=\tilde{\ell}$, the maximal ideals of those local rings satisfy $$\mathfrak{m}_{X'_{\tilde{\ell}\perf}, y}=\mathfrak{m}_{X'_{\tilde{\ell}}, x}\otimes_{\tilde{\ell}}\,\tilde{\ell}\perf.$$ This shows in particular that $\mathfrak{m}_{X'_{\tilde{\ell}\perf}, y}$ is generated by two elements; indeed, $X'$, and hence also  $\Og_{X'_{\tilde{\ell}}, x},$ are known to be regular. We conclude that $\Og_{X'_{\tilde{\ell}\perf}, y}$ is regular, as required.
\end{proof}\\
We now show that elliptic K3 surfaces are minimal. Recall that an elliptic K3 surface over $k$ is an elliptic surface $f\colon X\to C$ over $k$ such that $\omega_{X/k}\cong \Og_X$ and $H^1(X, \Og_X)=0.$ 
\begin{proposition}
Let $X$ be a K3 surface over $k$ and let $k_Z$ be a finite extension of $k$. Let $Z:=\Ps^1_{k_Z}$ and let $\iota\colon Z\to X$ be a closed immersion over $k$. Then $\deg_{k_Z} \iota^\ast \Og_X(Z)\not=-1.$  \label{K3minlem}
\end{proposition}
\begin{proof}
Assume the Proposition is false. We first reduce to the case where $k_Z$ is purely inseparable over $k.$ As above, let $\ell$ be the separable closure of $k$ in $k_Z$ and consider the commutative diagram (\ref{CD}). The same argument then  shows that $\iota_{\ell}^\ast \Og_{X_{\ell}}(Z)\cong \iota^\ast \Og_X(Z).$

Therefore we may now assume that $\ell=k$, $\pi'=\mathrm{Id}_X$ and $\iota_\ell=\iota$. Again as above, we have $\pi^\ast\Og_X(Z)=\Og_{X_{k_Z}}(Z)^{\otimes[k_Z:k]}.$ The morphism $\iota_Z$ commutes with the canonical projections from $Z$ and $X_{k_Z}$ to $\Spec k_Z$, both of which are smooth. Now adjunction tells us that
$$
\omega_{Z/k_Z}^{\otimes[k_Z:k]} \cong \iota_Z^\ast \Big(\Og_{X_{k_Z}}(Z)\underset{{\Og_{X_{k_Z}}}}{\otimes} \omega_{X_{k_Z}/k_Z}\Big)^{\otimes[k_Z:k]} \cong \iota^\ast \Og_X(Z).
$$
Since $Z\cong \Ps^1_{k_Z}$, we obtain 
$$
-1=\deg_{k_Z} \iota^\ast \Og_X(Z) =[k_Z:k]\deg_{k_Z} \omega_{Z/k_Z} =-2[k_Z:k],$$
which is clearly impossible. 
\end{proof}
\subsection{Weierstrass models}
In this section, we once again assume that $k$ is separably closed. We will recall (and refine) some basic results on Weierstrass models of elliptic curves, as well as their discriminants.
\begin{definition} \rm(Compare \cite[Definition 2.1]{Con}) \it
Let $R$ be a discrete valuation ring, and let $\eta$ be the generic point of $\mathrm{Spec}\,R$. A \rm planar integral Weierstrass model \it (or \rm Weierstrass model \it for short) of an elliptic curve $E$ over $\eta$ is a pair $(W, i)$ consisting of a closed subscheme $W\subseteq \Ps^2_R$ given by an equation
$$y^2z+a_1xyz+a_3yz^2-x^3-a_2x^2z-a_4xz^2-a_6z^3=0,$$
together with an isomorphism $i\colon E\to W_\eta$ sending the unit section of $E$ to $[0:1:0].$ \label{Weierstrassdef}
\end{definition}
To $(W, i)$ as above, we associate the elements \begin{eqnarray*} b_2 & := & a_1^2+4a_2 \\ b_4 &:=& 2a_4+a_1a_3 \\  b_6 &:=& a_3^2+4a_6 \\ b_8 &:=& a_1^2a_6+4a_2a_6-a_1a_3a_4+a_2a_3^2-a_4^2\end{eqnarray*} of $R,$ as well as the \it discriminant \rm 
$$\Delta:=-b_2^2b_8-8b_4^3-27b_6^2+9b_2b_4b_6.$$
It follows from \cite[Corollary 2.9]{Con} that any Weierstrass model of $E$ which is isomorphic to $W$ as a model of $E$ yields a discriminant which is an $R^\times$-multiple of $\Delta$. Hence the valuation of $\Delta$ only depends on the pair $(W, i)$ and not on the particular equation chosen. 
\begin{definition}
A Weierstrass model $(W, i)$ of an elliptic curve $E$ over $R$ is \rm minimal \it if its discriminant has minimal valuation among all possible Weierstrass models of $E.$  
\end{definition}

There is only one minimal Weierstrass model of a given elliptic curve, up to (unique) isomorphism of models; the minimal proper regular model of $E$ is equal to the minimal regular resolution of the minimal Weierstrass model (see \cite[Corollaries 4.6 and 4.7]{Con}). 

The following result is an immediate consequence of Szyd\l o's recent extension of Tate's algorithm to the case of potentially imperfect residue fields \cite{Szy}:
\begin{lemma} \label{Szydlo}
Let $R$ be a discrete valuation ring, and let $\eta$ and $s$ denote the generic and closed points respectively of $\mathrm{Spec}\,R$. Assume that $\kappa(s)$ is separably closed, of characteristic neither $2$ nor $3$. Let $E$ be an elliptic curve over $\eta$ with minimal proper regular model $\mathscr{E}\to \Spec R.$ Let $\nu$ be the valuation of the discriminant of a minimal Weierstrass model of $E$ over $R.$ Finally, let $m$ be the number of geometric irreducible components of $\mathscr{E}_s$, counted without multiplicities, and let 
$$\epsilon:=\begin{cases}-1 &\text{if $E$ has good reduction},\\
0& \text{if $E$ has multiplicative reduction},\\
1& \text{if $E$ has additive reduction}.\end{cases}$$
Then $m=\nu-\epsilon.$
\end{lemma}
\begin{proof}
This follows from \cite[Lemma 1, Tables 1 and 3]{Szy}.
\end{proof}\\
As usual, we define the terms \it good reduction \rm etc. to mean the following: Suppose $\mathscr{M}$ is the N\'eron model of $E$ over $R$ with identity component $\mathscr{M}^0.$ Choose an algebraic closure $\kappa(s)\alg$ of the residue field of $R.$ Then $E$ has \it good reduction \rm (resp. \it multiplicative reduction, additive reduction\rm) if $\mathscr{M}^0\times_{\Spec R}\Spec \kappa(s)\alg$ is an elliptic curve (resp. isomorphic to $\Gm$, $\mathbf{G}_{\mathrm{a}}$) over $\kappa(s)\alg$. 
\begin{lemma} In the setting of Lemma \ref{Szydlo}, we have $\nu = 0$ if and only if $\mathscr{E}$ is $R$-smooth. \label{DeltaElem}
\end{lemma}
\begin{proof} Let $W$ denote a minimal Weierstrass model of $E$. By \cite[Corollary 4.7]{Con}, we know that there is a unique morphism of models $\mathscr{E}\to W$, which turns $\mathscr{E}$ into the minimal regular resolution of $W$. A well-known calculation shows that $W$ is smooth if and only if $\nu=0$. Hence if $\nu=0$, it follows immediately that $\mathscr{E}\to W$ is an isomorphism, and hence that $\mathscr{E}$ is smooth over $R$. Conversely, if $\mathscr{E}$ is smooth over $R$, then \cite[Theorem 5.4]{Con} implies that the canonical morphism $\mathscr{E} \to \mathscr{M}$ from $\mathscr{E}$ to the N\'eron model $\mathscr{M}$ of $E$ over $R$ is an isomorphism. In particular, $\mathscr{M}_s$ is connected, and \cite[Theorem 5.5]{Con} implies that $W$ is isomorphic to $\mathscr{M}$, so that $W$ is $R$-smooth. We have already observed that this implies $\nu=0.$ 
\end{proof}
\begin{lemma}
In the setting of Lemma \ref{Szydlo}, let $(W,i)$ be a Weierstrass model of $E$. Assume that there is a morphism $\mathscr{E}\to W$ of models. Then $(W,i)$ is a minimal Weierstrass model.  \label{Wminimallem}
\end{lemma}
\begin{proof}
The morphism $\pi\colon \mathscr{E}\to W$ is certainly proper and birational. Furthermore, its fibres cannot contain any exceptional curves: $\mathscr{E}$ is minimal by assumption, and therefore its special fibre does not contain any exceptional curves at all. Hence $\pi$ turns $\mathscr{E}$ into the minimal regular resolution of $W$, by \cite[Theorem 2.2.2]{CES}; here we use that $W$ is normal, which follows from \cite[Lemma 2.3]{Con}. The result now follows from \cite[Corollary 4.7]{Con}.
\end{proof}

We are now ready to prove the main result of this section; recall that $\mathrm{char}\,k \not\in \{2,3\}$.
\begin{theorem}
Let $f\colon X\to C$ be a minimal elliptic surface over $k$. Let $x\in C$ be a closed point such that the fibre of $f$ above $x$ is singular. Then $\kappa(x)=k$. \label{separablethm}
\end{theorem}
\begin{proof}
Let $\overline{x}$ be the unique preimage of $x$ under the projection $C_{k\alg}\to C$, and let $\Og_{C,x}\subseteq \Og_{C_{k\alg},\overline{x}}$ be the corresponding extension of discrete valuation rings. We shall prove that the quantities $m$ and $\epsilon$ (see Lemma \ref{Szydlo}) do not change when passing from $k$ to $k\alg$, whereas $\nu$ is replaced by $[\kappa(x):k]\nu.$ This will force $\nu=0$ or $\kappa(x)=k$. More precisely, let $f_{k\alg}\colon X_{k\alg}\to C_{k\alg}$ be the induced elliptic surface over $k\alg.$ Let $X_{\Og_{C,x}}:=\Spec\Og_{C,x}\times_{C}X$ and $X_{\Og_{C_{k\alg}, \overline{x}}}:=\Spec \Og_{C_{k\alg},\overline{x}}\times_{C_{k\alg}} X_{k\alg}.$ Then there is a natural isomorphism 
\begin{align}X_{\Og_{C_{k\alg},\overline{x}}}=(\Spec \Og_{C_{k\alg},\overline{x}})\times_{\Spec \Og_{C,x}}X_{\Og_{C,x}}.\label{Xbasechange}\end{align} 

By Proposition \ref{generalminprop}, we know that $X_{\Og_{C,x}}$ and $X_{\Og_{C_{k\alg},\overline{x}}}$ are the minimal proper regular models of their respective generic fibres. Let $m, \epsilon, \nu$ be the quantities from Lemma \ref{Szydlo} associated with $X_{\Og_{C,x}}$ and let $m', \epsilon', \nu'$ be the analogous quantities for $X_{\Og_{C_{k\alg},\overline{x}}}.$ We then have $m = m'$, since $\Spec k\alg\to \Spec \kappa(x)$ is a universal homeomorphism.

Next, we observe that $X_{\Og_{C_{k\alg}, \overline{x}}}\to X_{\Og_{C,x}}$ is a homeomorphism. Hence there exists a unique open $U\subseteq X_{\Og_X}$ such that $U\times_{\Spec\Og_{C,x}}\Spec \Og_{C_{k\alg},\overline{x}}$ is the smooth locus of $X_{\Og_{C_{k\alg},\overline{x}}}$ over $\Og_{C_{k\alg},\overline{x}}.$ Because smoothness is fpqc-local on the target, it follows that $U$ is contained in the smooth locus of $X_{\Og_{C,x}}$ over $\Og_{C,x}.$ On the other hand, this locus is clearly contained in $U$, so we have an equality. Together with \cite[Theorem 5.4]{Con}, this implies that $\epsilon=\epsilon'.$ 

Finally, let us denote by $(W, i)$ a minimal Weierstrass model of the generic fibre of $X_{\Og_{C,x}}.$ After choosing an equation for $W$ as in Definition \ref{Weierstrassdef}, we may calculate the corresponding discriminant $\Delta\in \Og_{C,x}$. Clearly, $W\times_{\Spec\Og_{C,x}}\Spec \Og_{C_{k\alg},\overline{x}}$ is still a Weierstrass model of the generic fibre of $X_{\Og_{C_{k\alg},\overline{x}}}$. Our chosen equation gives an equation for $W\times_{\Spec\Og_{C,x}}\Spec \Og_{C_{k\alg},\overline{x}}$ as well, and the corresponding discriminant is simply the image of $\Delta$ in $\Og_{C_{k\alg},\overline{x}}$. It now follows from \cite[Corollary 4.7]{Con} that there is a proper birational map $X_{\Og_{C,x}}\to W$ which respects the identifications at the generic fibre. Hence there also exists a map $X_{\Og_{C_{k\alg},\overline{x}}}\to W\times_{\Spec\Og_{C,x}}\Spec \Og_{C_{k\alg},\overline{x}}$ with the same properties. Lemma \ref{Wminimallem}  now tells us that $W\times_{\Spec\Og_{C,x}}\Spec\Og_{C_{k\alg},\overline{x}}$ is a minimal Weierstrass model. Using Lemma \ref{Oramlem}, we conclude that $\nu'=[\kappa(x):k]\nu.$ This forces $\nu=0$ or $[\kappa(x):k]=1$. Since the former would contradict the assumption that $f^{-1}(x)$ be singular (see Lemma \ref{DeltaElem}), we conclude that $\kappa(x)=k$. 
\end{proof}\\
$\mathbf{Remark}.$ The requirement that $f\colon X\to C$ be minimal is, in fact, redundant. Indeed, if $f$ is not minimal and $Z\subseteq X$ is an exceptional curve contained in some fibre of $f$, then $Z\cong \Ps^1_k$. Indeed, we know that $Z\cong \Ps^1_{k_Z}$ for some finite extension $k_Z$ of $k$, and the proof of Proposition \ref{generalminprop} shows that $k_Z$ is separable over $k$. This shows that if $\tilde{f}\colon X^{\mathrm{min}}\to C$ is the minimal proper regular model of the generic fibre of $f$ and $x\in C$ is a closed point such that the canonical morphism $X\to X^{\mathrm{min}}$ does not induce an isomorphism of fibres above $x$ then $\kappa(x)=k.$ In particular, if $\tilde{f}^{-1}(x)$ is smooth but $f^{-1}(x)$ is not, then $\kappa(x)=k.$\\
\\
$\mathbf{Remark.}$ Let $\eta$ be the generic point of $C.$ Theorem \ref{separablethm} can be interpreted as a statement about elliptic curves over $\kappa(\eta)$. Indeed, if $E$ is such an elliptic curve, we can construct the minimal proper regular model $\mathscr{E}\to C$ of $E$. If the ground field is not perfect, there is nothing which guarantees that $\mathscr{E}\to \Spec k$ be smooth. The previous result gives us a method to construct examples where this is not the case. Let $p>3$ be the characteristic of $k$ and assume that $a\in k$ is not a $p$-th power. Let $C=\Ps^1_k$ and let $E$ be the elliptic curve given by the minimal Weierstrass equation $y^2=x^3+(t^p-a)$. Clearly, $E$ has reduction of type $II$ at the closed point $x$ of $\mathbf{A}^1_k\subseteq \Ps^1_k$ given by the ideal $\langle t^p-a\rangle \subseteq k[t].$ Since the residue field $\kappa(x)=k(a^{1/p})$ is not separable over $k$, Theorem \ref{separablethm} tells us that $\mathscr{E}$ is not smooth over $k$. 

\subsection{Global Weierstrass models}
In this section, we denote by $k$ an arbitrary field, and by $f\colon X\to C$ a minimal elliptic surface over $k$, equipped with a fixed section $\sigma\colon C\to X$. As before, we denote by $\eta$ the generic point of $C$. We then have the following.
\begin{proposition}
There exists a unique pair $(W, \pi)$, with $w:W\to C$ projective and $\pi\colon X\to W$ birational, such that for each closed point $x \in C$, the pair $(W\times_{C}\Spec \Og_{C,x}, \pi_\eta)$ is a minimal Weierstrass model of $X\times_{C} \Spec \kappa(\eta)$. Furthermore, the scheme $W$ is normal, and the fibres of $W \to C$ are geometrically integral and have at most one geometric singularity. \label{globalweierstrassprop}

\end{proposition}
\begin{proof}
Let $x_1,\dots, x_n$ be the closed points of $C$ above which the fibre of $f$ is singular, and let $U:=C\setminus \{x_1,\dots, x_n\}.$ For each $j=1,\dots, n,$ let $w_j\colon W_j\to \Spec \Og_{C,x_j}$ be the minimal Weierstrass model of $X_\eta$, and let $W_0:=f^{-1}(U)$. The section $\sigma\colon C\to X$ induces sections $\sigma_j\colon \Spec \Og_{C,x_j}\to W_j$ and $\sigma_0\colon U\to W_0$. By abuse of notation, we shall denote the images of the sections by the same symbol as the sections themselves. 

The line bundles $\mathscr{L}_j:=\Og_{W_j}(\sigma_j)$ are relatively ample; this follows from the proof of \cite[Theorem 2.8]{Con} for $j=1,\dots, n$ and is well-known for $j=0$. Let $U_0,\dots,U_n$ denote the schemes $U,\Spec \Og_{C,x_1},\dots,\Spec\Og_{C,x_n}$ respectively. For each pair of integers $0\leq i,j\leq n$, denote by $U_{ij}$ the scheme $U_i\times_{C} U_j.$ As soon as $i\not=j,$ there is a canonical isomorphism
$$U_{ij}=\Spec \kappa(\eta),$$ which means that, in this case, $W_i\times_{C}U_{ij}$ can be canonically identified with $X_\eta$, and the pullback of $\mathscr{L}_j$ to this scheme is canonically isomorphic to $\Og_{X_{\eta}}(\sigma_{\eta}).$ Hence we obtain a descent datum of pairs $(W_i, \mathscr{L}_i)$ with respect to the fpqc covering $U_0\sqcup U_1\sqcup \cdots \sqcup U_n\to C$. By \cite[\S6.1, Theorem 7]{BLR}, this descent datum is effective, which settles the existence part.

 For the uniqueness, we simply observe that Lemma \ref{Wminimallem} implies that the $X\times_{C}\Spec \Og_{C,x}$ are the minimal Weierstrass models of their respective generic fibres. Therefore the descent datum obtained from $W\to C$ with respect to the fpqc covering described above is canonically isomorphic to the one we used previously to prove existence.

To see that $W$ is normal, let $t$ be a (set-theoretic) point of $W.$ If $t$ maps to $U_0,$ then $W$ is smooth at $t$. If $t$ maps to one of the $x_j,$ it suffices to observe that the $W_j$ are normal \cite[Lemma 2.3]{Con}. The last claim follows from \cite[Lemma 2.3]{Con}
\end{proof}\\
We shall refer to $W$ as \it the Weierstrass model of $f\colon X\to C$. \rm 
\begin{proposition}
Let $w\colon W\to C$ be the Weierstrass model of $X.$ Let $x_1,\dots, x_r$ be the closed points of $C$ above which the fibre of $f$ has at least two irreducible components. For $j=1,\dots,r$, there exists a unique closed point $z_j$ of $W$ above $x_j$ where $W$ is singular, and the extension $k\subseteq \kappa(z_j)$ is purely inseparable. If $\mathrm{char}\, k\not\in \{2,3\},$ then $\kappa(z_j)=k$ for all $j.$

Furthermore, the map $\pi\colon X\to W$ is an isomorphism above $W\setminus \{z_1,\dots, z_r\}$, and $\pi^{-1}(z_j)$  is set-theoretically the union of the irreducible components of $f^{-1}(x_j)$ which do not meet $\sigma$. 
\end{proposition}
\begin{proof}
We already know that $X_{\Og_{C,x}}\to W_x=W\times_{C}\Spec \Og_{C,x}$ is a minimal regular resolution for all closed points $x\in C$. In particular, if $f^{-1}(x)$ is irreducible, then $W_x$ must be regular. On the other hand, one can check easily (using equations) that $w^{-1}(x_j)$ has precisely one point $z_j$ which is not smooth over $\kappa(x_j)$. Clearly, $W_{x_j}$ is regular at all other points. If $W$ were regular at $z_j$, then the morphism $X_{\Og_{C,x_j}}\to W_{x_j}$ would have to be an isomorphism, which contradicts our assumption on $f^{-1}(x_j)$. The morphism $X\to W$, being a minimal regular resolution, is an isomorphism above the regular locus of $W$. The remaining claims follow from \cite{Con}, more precisely from the proof of Corollary 4.7 and the discussion after the proof of Theorem 5.5 of \it loc. cit.\rm
\end{proof}
\begin{lemma}
Let $W\to C$ be the Weierstrass model of $X.$ Then $W\times_{\Spec k}\Spec k\alg$ is the Weierstrass model of $\overline{f}\colon X_{k\alg}\to C_{k\alg}.$ \label{WBasechangelem}
\end{lemma}
\begin{proof}
Clearly $W_{k\alg}\to C_{k\alg}$ still has the property that its localisations at closed points of $C_{k\alg}$ are Weierstrass models of their generic fibres; this can be checked using equations. Furthermore, we still have a birational morphism $X_{k\alg}\to W_{k\alg}$ which commutes with the projections to $C_{k\alg}.$ Hence Lemma \ref{Wminimallem} in combination with Proposition \ref{globalweierstrassprop} implies the result. 
\end{proof}

\section{Tameness of the discriminant locus}
Now let $K$ be a complete discretely valued field with ring of integers $\Og_K$ and algebraically closed residue field of characteristic not dividing 6. Let $f\colon X\to C$ be a minimal elliptic surface over $K.$ We keep the notation from the previous sections; in particular, $w\colon W\to C$ will be the global Weierstrass model of $X$, $\pi\colon X\to W$ the corresponding projection, and for a closed point $x\in C,$ we let $X_{\Og_{C,x}}:=X\times_{C} \Spec \Og_{C,x},$ and $W_x:=W\times_{C}\Spec \Og_{C,x}.$ An immediate consequence of our previous work is the following
\begin{proposition}
Let $x$ be a closed point of $C$ such that the fibre of $f$ above $x$ is singular. Then the extension $K\subseteq \kappa(x)$ is separable. \label{separableprop}
\end{proposition} 
\begin{proof}
This is an immediate consequence of Theorem \ref{separablethm} and Proposition \ref{generalminpropII}.
\end{proof}\\
We shall now introduce the cohomological conditions that will later be shown to guarantee logarithmic good reduction up to modification. First choose an algebraic closure $K\alg,$ and let $K\sep$ be the separable closure of $K$ in $K\alg.$ Then the group $\Gal(K\sep/K)$ acts on $K\alg$ in a canonical way. Furthermore, for any $\ell$ invertible in $\Og_K,$ the $\Q_\ell$-vector spaces $H^i(X_{K\sep}, \Q_\ell)$ and $H^i(X_{K\alg}, \Q_\ell)$ are canonically identified; the same is true for cohomology with coefficients $\Z/\ell^n\Z$ (this follows from the topological invariance of \'etale cohomology). In particular, we obtain representations of $\Gal(K\sep/K)$ on $H^i(X_{K\alg}, \Q_\ell)$ (resp. $H^i(X_{K\alg}, \F_\ell)$). These coincide with the (more commonly used) representations on $H^ i(X_{K\sep}, \Q_\ell)$ (resp. $H^i(X_{K\sep}, \F_\ell)$), but our definition will turn out to be more convenient for our purposes. \\
Let $f'\colon X'\to C$ be an elliptic surface over $K$ which is birational to $X,$ and which has the property that \it each singular fibre of $f'$ has at least two geometric irreducible components. \rm As such models will feature repeatedly in what follows, we make the following
\begin{definition}
An elliptic fibration $f'\colon X'\to C$ over $K$ is called \rm big \it if for each closed point $x$ of $C$ such that $f'$ is not smooth above $x,$ the geometric fibre of $f'$ above $x$ has at least two components. 
\end{definition}
Note that each (not necessarily minimal) elliptic surface $f\colon X\to C$ admits a projective modification $X'\to X$ which is big; to construct one, simply blow up the point hit by the section in each singular fibre of $f$ which has only one geometric irreducible component. Note that the additional component in the singular fibres constructed this way have multiplicity one and are geometrically integral. We can now formulate the first cohomological condition:
\begin{itemize}
\item Condition 1: \it There exists a big elliptic fibration $f'\colon X'\to C,$ birational (over $C$) to $X,$ such that for all $i\geq 0,$ the Galois representations $H^i(X'_{K\alg}, \Q_\ell)$ are tamely ramified. \label{Condition1}
\end{itemize}
To formulate the second cohomological condition, we shall first study the sheaf $\R^1f_\ast \F_\ell$. After proving some auxiliary lemmata, we shall see that this sheaf is canonically isomorphic to the $\ell$-torsion subsheaf of the identity component of the N\'eron model $\mathscr{M}$ of the generic fibre of $f.$ 
\begin{lemma}
The functor $\Pic^0_{W/C}$ is representable by a smooth group scheme of finite presentation over $C$ with connected fibres, and the canonical morphism $\Pic^0_{W/C}\to \mathscr{M}^0$ is an isomorphism. The same statements hold if we replace $W$ by $W_{K\alg}$ and $C$ by $C_{K\alg}.$ 
\end{lemma}
\begin{proof}
The claim regarding the representability of $\Pic^0_{C/K}$ follows from \cite[Chapter 9.3, Theorem 1]{BLR}. The fibres of $\Pic^0_{C/K}$ are then connected by definition. For the second claim we observe that the N\'eron model commutes with localisations at closed points of $C,$ as does $W$ (by construction). Suppose $U\subseteq C$ is the largest open subset of $C$ such that $f$ is smooth above $U.$ Then $U$ is non-empty, and we let $x_1,..., x_n$ be the complement of $U.$ It suffices to show that the morphism $\Pic^0_{W/C}\to \mathscr{M}^0$ induces isomorphisms of the respective localisations at the points $x_1,$..., $x_n.$ This, however, follows from the fact that minimal Weierstrass models are normal and have rational singularities (see \cite[Corollary 8.4]{Con}), together with \cite[Chapter 9.3, Theorem 1]{BLR}. The remaining claim now follows from Lemma \ref{WBasechangelem}.
\end{proof}\\
Now let $\pi'\colon X'\to W$ be the composition of the canonical morphism $X'\to X$ with the projection $\pi\colon X\to W.$ Recall that we have a commutative diagram
$$\begin{tikzcd}
X' \arrow{r}{\pi'}\arrow[swap]{rd}{f'} &W \arrow{d}{w}\\
&C.
\end{tikzcd}$$
We denote by $\overline{\pi}',$ $\overline{f}',$ and $\overline{w}$ the morphisms obtained from those in the diagram by base change to $K\alg.$
\begin{lemma}
There is an isomorphism $\R^1\overline{w}_\ast \F_\ell\to \R^1\overline{f}'_\ast\F_\ell$ which is equivariant with respect to the canonical action of $\Gal(K\sep/K)$ on both sheaves. \label{treelem} 
\end{lemma}
\begin{proof}
The Leray spectral sequence for the composition $\overline{f}'=\overline{w}\circ \overline{\pi}'$ gives us a short exact sequence
$$0\to \R^1\overline{w}_\ast (\overline{\pi}'_\ast\F_\ell)\to \R^1 \overline{f}'_\ast \F_\ell\to \overline{w}_\ast \R^1\overline{\pi}'_\ast \F_\ell.$$ We observe that the morphism $\F_\ell\to \overline{\pi}'_\ast \F_\ell$ is an isomorphism, so we have at least already constructed a morphism with the correct source and target. We also know already that this map is injective. The morphism $\overline{\pi}'\colon X_{K\alg}\to W_{K\alg}$ is an isomorphism outside finitely many points of $W_{K\alg}\to C_{K\alg},$ and the reduced preimages of those points are trees of copies of $\Ps^1_{K\alg}.$ Hence the proper base change theorem tells us that the stalk of $\R^1\overline{\pi}'_\ast \F_\ell$ is zero at those points, and it is certainly zero at all other points as well. Hence the morphism we constructed above is indeed an isomorphism. The claim about equivariance follows from the fact that $\overline{f}',$ $\overline{w},$ and $\overline{\pi}'$ are compatible with the actions of $\Gal(K\sep/K)$ on the schemes $X'_{K\alg},$ $W_{K\alg},$ and $C_{K\alg}.$ 
\end{proof}
\begin{lemma}
There is an isomorphism $\R^1\overline{f}'_\ast \F_\ell \to \Pic^0_{W_{K\alg}/C_{K\alg}}[\ell]$ which is $\Gal(K\sep/K)$-equivariant. 
\end{lemma}
\begin{proof}
Since the map $\Gm\to \overline{w}_\ast \Gm$ is an isomorphism on the small \'etale site of $C_{K\alg},$ we see that the Kummer sequence gives us a short exact sequence
$$0\to \R^1\overline{w}_\ast \F_\ell \to \R^1\overline{w}_\ast \Gm \overset{[\ell]}\to \R^1\overline{w}_\ast\Gm$$ Since the (geometric) fibres of $\overline{w}$ are integral curves, it follows that $\Pic^0_{W_{K\alg}/C_{K\alg}}$ is the same as the kernel of the degree map $\Pic_{W_{K\alg}/C_{K\alg}} \to \Z,$ which implies 
\begin{align*}
\Pic^0_{W_{K\alg}/C_{K\alg}}[\ell]&=\R^1 \overline{w}_\ast \Gm[\ell]\\
&=\R^1\overline{w}_\ast \F_\ell\\
&=\R^1\overline{f}'_\ast\F_\ell
\end{align*}
on the small \'etale site of $C_{K\alg},$ where the last equality is due to the preceding Lemma. The second claim follows from the fact that both $\overline{w}$ and $\overline{f}'$ are equivariant with respect to the action of $\Gal(K\sep/K)$ on $W_{K\alg}$ and $C_{K\alg},$ together with (again) the preceding Lemma.
\end{proof}\\
The same argument shows that there is an isomorphism 
$$\R^1f'_\ast \F_\ell\to \mathscr{M}^0[\ell];$$ in fact, the morphism constructed in the Lemma is the pull-back of this map. As a corollary to this observation, we can construct an integral regular curve $C'$ of finite type over $K$ together with a flat and generically \'etale morphism $\epsilon\colon C'\to C,$ such that, at the generic point of $C',$ the sheaf $\R^1f'_\ast\F_\ell$ is constant (and isomorphic to $\F_\ell^2$).  If we denote by $\overline{\epsilon}\colon C'_{K\alg}\to C_{K\alg}$ the induced morphism, we are now ready to formulate our second cohomological condition: 
\begin{itemize}
\item Condition 2: \it The Galois representation on $H^1(C_{K\alg}, \F_\ell\oplus \overline{\epsilon}_\ast\overline{\epsilon}^\ast\R^1\overline{f}'_\ast\F_\ell)$ is tamely ramified. \label{Condition2}
\end{itemize}
For later use, we shall prove now a Lemma about the curve $C':$ 
\begin{lemma}
The curve $C'_{K\alg}:=C'\times_{\Spec K}\Spec K\alg$ is reduced and equal to the scheme-theoretic disjoint union of its irreducible components. \label{CDecomplem}
\end{lemma}
\begin{proof}
Because the map $\epsilon$ is generically smooth, we know that so is the morphism $C'\to \Spec K.$ Hence there exists a non-empty (hence dense) open subscheme $U\subseteq C'$ which is smooth over $K.$ For any open subscheme $V\subseteq C'$, the map $\Og_{C'}(V)\to \Og_{C'}(U\cap V)$ is injective, and remains so after base change to $K\alg.$ But the target is geometrically reduced, so the first claim follows. The second claim is purely topological; since the map $\Spec K\alg\to \Spec K\sep$  is a universal homeomorphism, it suffices to prove the corresponding claim for $C'\times_{\Spec K} \Spec K\sep.$ But this scheme is regular (as $C'$ is regular and $K\subseteq K\sep$ is separable), so no two irreducible components can intersect non-trivially. 
\end{proof}
\begin{theorem}
Let $f\colon X\to C$ be a minimal elliptic surface over $K$ which satisfies one of Condition 1 or Condition 2 above. Let $x\in C$ be a closed point above which the fibre of $f$ is singular. Then the extension $K\subseteq \kappa(x)$ is tamely ramified.  \label{Newtamethm}
\end{theorem}
\begin{proof}
The proof of this theorem will occupy the remainder of this paragraph. Suppose first that \it Condition 2 \rm is satisfied, and assume that $K\subseteq \kappa(x)$ is wildly ramified. Suppose that $\eta$ is the generic point of $C',$ and let $\gamma\colon \Spec (\kappa(\eta)\otimes_KK\alg)\to C'_{K\alg}$ be the canonical map. We have a Galois-equivariant exact sequence
$$0\to \overline{\epsilon}^\ast \R^1\overline{f}_\ast \F_\ell\to \gamma_\ast\gamma^\ast \overline{\epsilon}^\ast\R^1\overline{f}_\ast \F_\ell\to \bigoplus_{x\in S} x_\ast M_x\to 0,$$
where $S$ is the set of $K\alg$-points of $C'_{K\alg}$ above which the order of the stalk of $\overline{\epsilon}^\ast\R^1\overline{f}_\ast \F_\ell$ drops. Note that $M_x$ is isomorphic to either $\F_\ell$ or $\F_\ell^2.$ Now write
$$C'_{K\alg}=C_1\sqcup...\sqcup C_d$$ for the decomposition of $C'_{K\alg}$ into irreducible components. We know from Lemma \ref{CDecomplem} that this is possible. We now claim that, on each $C_j,$ there exist at least three $K\alg$-points at which the order of the stalk of $\overline{\epsilon}^\ast\R^1\overline{f}_\ast \F_\ell$ drops. This follows because the corresponding statement for $C_{K\alg}$ is true, which, in turn, follows from the fact that the residue characteristic of $K$ is at least three and that $K\subseteq \kappa(x)$ is wildly ramified. This allows us to deduce that the action of $\Gal(K\sep/K)$ on the set $\{C_1,..., C_d\}$ is tame. Indeed, the exact sequence above shows that, for each $j=1,...,d,$ we have
$$H^1(C_j, \overline{\epsilon}^\ast\R^1\overline{f}_\ast \F_\ell\mid_{C_j})\not=0.$$ Hence we can find a torsor $\mathscr{T}$ on $C'_{K\alg}$ which is non-trivial on $C_j$ but trivial on all other components. As Condition 2 guarantees that the Galois action $H^1(C'_{K\alg}, \overline{\epsilon}^\ast\R^1\overline{f}_\ast \F_\ell)$ is tamely ramified, this means that all elements of the wild inertia subgroup must fix the component $C_j.$ This observation implies that the Galois action on $\Gamma(C'_{K\alg}, \gamma_\ast\gamma^\ast \overline{\epsilon}^\ast\R^1\overline{f}_\ast \F_\ell)$ is tamely ramified. Now we consider the exact sequence
$$\Gamma(C'_{K\alg}, \gamma_\ast\gamma^\ast \overline{\epsilon}^\ast\R^1\overline{f}_\ast \F_\ell)\to \Gamma(C'_{K\alg}, \bigoplus_{x\in S} x_\ast M_x)\to H^1(C'_{K\alg}, \overline{\epsilon}^\ast\R^1\overline{f}_\ast \F_\ell).$$ All groups in this sequence are $\F_\ell$-vector spaces, so taking invariants under the wild inertia group does not destroy exactness of this sequence. This, however, means that the Galois action on $\Gamma(C'_{K\alg}, \bigoplus_{x\in S} x_\ast M_x)=\bigoplus_{x\in S} M_x$ is tamely ramified, so the same must be true for the Galois action on $S.$ However, the set of $K\alg$-points of $C$ above which the fibre of $f$ is singular is exactly the image of $S$ under $\overline{\epsilon},$ so our claim follows. \\
Now suppose that \it Condition 1 \rm is satisfied. Some of the arguments that follow are inspired by Paragraph 3 of \cite{Sch}; we provide some details not covered in \it loc. cit., \rm particularly relating to Galois equivariance. Until the end of this proof, we impose the condition that $\overline{f}$ not be smooth, i.e., that the minimal elliptic fibration $\overline{f}\colon X_{K\alg}\to C_{K\alg}$ have at least one singular fibre, which clearly causes no loss of generality. 
First note that we have a Galois equivariant spectral sequence
$$E_2^{i,j}=H^i(C_{K\alg}, \R^j \overline{f}'_\ast \Q_\ell)\Rightarrow H^{i+j}(X'_{K\alg}, \Q_\ell).$$ 
\begin{lemma}
The spectral sequence above has the property that $E_2^{0,2}=E_{\infty}^{0,2}.$ \label{SpectralLemmaI}
\end{lemma}
\begin{proof}
It follows from general principles that $E_3^{0,2}=E_{\infty}^{0,2},$ and that $E_{3}^{0,2}=\ker(E_2^{0,2}\to E_2^{2,1}).$ The proof of \cite[Proposition 3.1]{Sch} shows that 
$$E_2^{2,1}=(\varprojlim \Hom(M[\ell^n], \Z/\ell^n\Z))\otimes_{\Z_\ell}\Q_\ell,$$
where $M$ denotes the Mordell-Weil group of $\overline{f}\colon X_{K\alg}\to C_{K\alg}.$ However, by \cite[Theorem 1.1]{Shi}, $M$ is a finitely generated Abelian group (this is where we use the condition on the existence of a singular fibre of $\overline{f}$; otherwise, the Theorem from \it loc. cit. \rm would not apply). Hence $E_2^{2,1}=0,$ which implies our claim. 
\end{proof}\\
We shall now prove the formula
\begin{align}
E_2^{0,2}= \Q_\ell \oplus \bigoplus_{ \text{$Y$ exceptional}} \Q_\ell\langle Y \rangle, \label{formula}
\end{align}
where the index set of the direct sum is meant to be the set of irreducible components of the exceptional locus of the canonical map $X'_{K\alg}\to X_{K\alg}\to W_{K\alg}.$ The action of some $\tau\in \Gal(K\sep/K)$ is given by sending such a $Y$ to $\tau^{-1}(Y).$ 
Recall that there is a Galois equivariant spectral sequence
$$E_2^{i,j}=\R^i \overline{w}_\ast \R^j \overline{\pi}'_\ast \Z/\ell^n\Z\Rightarrow \R^{i+j}\overline{f}'_\ast \Z/\ell^n\Z.$$ We have the following
\begin{lemma}
The spectral sequence above satisfies $E_2^{i,j}\not=0$ only if $ij=0$ and $i+j\leq 2.$ Moreover, the spectral sequence stabilizes at the $E_2$-page.  
\end{lemma}
\begin{proof}
The map $\overline{\pi}'$ is a birational morphism of surfaces, so $\R^j\overline{\pi}'_\ast \Z/\ell^n\Z$ is concentrated at finitely many points as soon as $j\geq1.$ This implies the claim that $E_2^{i,j}=0$ as soon as $ij\not=0.$ The claim that $E_2^{i,j}$ vanishes for $ij=0,$ $i+j\geq 3$ follows from the fact that both $\overline{\pi}'$ and $\overline{w}$ have fibres of dimension at most one. To see that the spectral sequence stabilises at $E_2,$ we consider the only differential map (on any page) whose source and target are both non-zero. This is the morphism 
$$\overline{w}_\ast \R^1\overline{\pi}'_\ast\Z/\ell^n\Z\to \R^2\overline{w}_\ast (\overline{\pi}'_\ast\Z/\ell^n\Z).$$ However, the source of this map vanishes by the same argument as in the proof of Lemma \ref{treelem}.
\end{proof}\\
Continuing with our proof of Formula (\ref{formula}), we note that the Lemma gives us a Galois equivariant split exact sequence 
$$0\to \Z/\ell^n\Z\to \R^2\overline{f}'_\ast \Z/\ell^n\Z\to \overline{w}_\ast \R^2\overline{\pi}'_\ast\Z/\ell^n\Z\to 0$$ (for the splitting, see the proof of Lemma 3.2 in \cite{Sch}). This gives us an isomorphism
$$E_2^{0,2}=\Q_\ell \oplus (\varprojlim \Gamma(W_{K\alg}, \R^2 \overline{\pi}'_\ast\Z/\ell^n\Z))\otimes_{\Z_\ell}\Q_\ell.$$
Now let $S$ be the scheme consisting of the spectra of the residue fields of the closed points of $W_{K\alg}$ above which the map $\overline{\pi}'\colon X'_{K\alg}\to W_{K\alg}$ is not an isomorphism. Let $s\colon S_{K\alg}\to W_{K\alg}$ be the canonical closed immersion. Further let $X'_S$ be the pull-back of $X'_{K\alg}$ along $s.$ To ease notation, we note that the category of \'etale Abelian sheaves on $S$ is canonically equivalent to the category whose objects are Abelian groups together with a direct sum decomposition indexed by $S$; under this equivalence, morphisms correspond to decomposable homomorphisms of Abelian groups. Consider the sheaf $\bigoplus_Y \Z\langle Y\rangle$ on $S$ (the decomposition comes from the map which sends a component of the exceptional divisor to an element of $S$). There is a canonical action on this sheaf under which an element $\tau\in \Gal(K\sep/K)$ acts as $Y\mapsto \tau^{-1}(Y).$ By Kummer theory, we have a canonical Galois equivariant isomorphism
$$\R^2\overline{\pi}'_\ast \Z/\ell^n\Z=s_\ast s^\ast \R^2\overline{\pi}'_\ast\Z/\ell^n\Z=s_\ast (\Pic_{X_S/S}\otimes_{\Z}\Z/\ell^n\Z).$$ 
By \cite[Chapter 9.2, Corollary 14]{BLR} there is a Galois equivariant morphism 
$$\Pic_{X_S/S}\to \bigoplus_Y \Z\langle Y\rangle$$ whose kernel is $\Pic^0_{X_S/S}$ and whose image is a finite-index equivariant subsheaf $\Lambda\subseteq \bigoplus_Y \Z\langle Y\rangle.$ Hence we find $$\R^2\overline{\pi}'_\ast \Z/\ell^n\Z=s_\ast (\Lambda\otimes_{\Z} \Z/\ell^n\Z).$$ Taking limits and tensoring with $\Q_\ell,$ we see that
$$\Gamma(W_{K\alg}, \R^2\overline{\pi}'_\ast \Q_\ell)=\bigoplus_{\text{$Y$ exceptional}} \Q_\ell\langle{Y}\rangle$$ as Galois representations, hence proving Formula (\ref{formula}). Now suppose $\Delta$ is the discriminant locus of the minimal elliptic fibration $f\colon X\to C,$ and let $\Delta'\subseteq C_{K\alg}$ be the set of closed points above which $\overline{f}'$ is singular. Then we have $\Delta\times_{\Spec K}\Spec K\alg\subseteq \Delta',$ and because the fibration $f'\colon X'\to C$ is big, the map $S\to \Delta'$ is surjective. Now Condition 1 tells us, together with Lemma \ref{SpectralLemmaI}, that the Galois action on the set of irreducible components of the exceptional divisor of $X'_{K\alg}\to W_{K\alg}$ is tame. Since this set surjects onto $S,$ we conclude that the Galois action on $\Delta\times_{\Spec K}\Spec K\alg$ is tamely ramified as well. This implies our claim. 
\end{proof}
\section{Removing horizontal log structure}
Let $g\colon X\to Y$ be a log smooth morphism of log schemes. Suppose that both $X$ and $Y$ are (classically) integral, and that the log structures $\mathscr{M}_X,$ $\mathscr{M}_Y$ on $X$ and $Y$ are the usual (compactifying) log structures associated with open subsets $U_X,$ $U_Y$ of $X$ and $Y$ whose complements are effective Cartier divisors. We further impose two conditions, namely that  \\
\\
(i) the scheme $X$ be regular in the classical sense, and the reduced complement of $U_X$ be a divisor with normal crossings, and that \\
\\
(ii) the log structure $\mathscr{M}_Y$ be fine and saturated (fs for short), i.e., it admit charts by finitely generated saturated integral monoids locally in the \'etale topology. \\
\\
Note that $\mathscr{M}_X$ is automatically fine, and that both log structures are fs. Suppose $D$ is the (reduced) complement of $U_X$ in $X.$ We can then write 
$$D=D_{\mathrm{hor}}+D_{\mathrm{vert}},$$
with $D_{\mathrm{hor}}$ and $D_{\mathrm{vert}}$ both effective, where $D_{\mathrm{hor}}$ is the sum of (some) irreducible components $D_0$ of $D$ such that $f(D_0)\cap U_Y \not=\emptyset$ (of course, such a decomposition is not unique).  Note that the complement of $D_{\mathrm{vert}}$ is contained in $g^{-1}(U_Y),$ and beware that $D_{\mathrm{hor}}$ need not be flat over $Y.$ We define a new log scheme $X',$ whose underlying scheme is given by that of $X,$ and whose log structure $\mathscr{M}_X'$ is associated with $U_{X'}=X\backslash D_{\mathrm{vert}}.$ Observe that the morphism of schemes underlying $g$ induces a morphism $g\colon X'\to Y$ of log schemes. The proof of the following Proposition seems to be similar, at least intuitively, to that of \cite[Proposition 3.1]{AK}, although the language and methods are somewhat different. Before we prove the Proposition, however, we shall need the following
\begin{lemma}
Let $\overline{x}$ be a geometric point of $X.$ Suppose that the pullback of $D$ to $\Spec \Og_{X, \overline{x}}$ is cut out by $x_1\cdot...\cdot x_d$ such that the sequence $x_1,...,x_d$ can be extended to a minimal set of generators of $\mathfrak{m}_{X, \overline{x}}.$ Then $\mathscr{M}_{X, \overline{x}}\subseteq \Og_{X, \overline{x}}$ is the submonoid generated by $x_1,..., x_d$ and $\Og_{X, \overline{x}}^\times.$ In particular, $\mathscr{M}_{X, \overline{x}}^{\mathrm{gp}}$ is canonically isomorphic to the subgroup of $(\Frac \Og_{X, \overline{x}})^\times$ generated by $x_1,...,x_d$ and $\Og_{X, \overline{x}}^\times.$ \label{stalklemma}
\end{lemma}
\begin{proof}
Let us denote the image of $\overline{x}$ in any scheme to which this geometric point maps by $x.$ We may assume without loss of generality that the pullback of $D$ to $\Spec \Og_{X, x}$ is cut out by $x_1'\cdot...\cdot x'_d,$ such that this sequence extends to a minimal set of generators of $\mathfrak{m}_{X,x}.$ Let $f\in \mathscr{M}_{X, \overline{x}}.$ Then $f$ is defined (as a function) on some \'etale neighbourhood $\overline{x}\to E\to X.$  Observe that both inclusions 
$$\Og_{X,x}\subseteq \Og_{E,x}\subseteq \Og_{X, \overline{x}}$$ have the property that the maximal ideal in the smaller local ring generates the maximal ideal in the bigger one (for the first inclusion, this follows because $E\to X$ is \'etale; for the second inclusion, because $\Og_{X,\overline{x}}$ is the strict Henselisation of $\Og_{E,x}$ with respect to the geometric point $\overline{x}\to E).$ 
Hence $x_1',...,x_d'$ remain prime elements in $\Og_{E,x}$ and $\Og_{X, \overline{x}}.$ We can view $f$ as an element of $\Og_{E,x}, $ and write
$$f=\epsilon x_1'^{\beta_1}\cdot...\cdot x_d'^{\beta_d},$$ with $\epsilon \in \Og_{E,x}^\times,$ in a unique way. This decomposition, and its uniqueness, remain valid in $\Og_{X, \overline{x}}.$ Up to ordering, we know that there are $\epsilon_1,..., \epsilon_d\in \Og_{X, \overline{x}}^\times$ such that $x'_j=\epsilon_jx_j$ for $j=1,..., d.$ Hence, in $\Og_{X, \overline{x}},$ we have 
$$f=\tilde{\epsilon}x_1^{\beta_1}\cdot...\cdot x_d^{\beta_d}$$ for some $\tilde{\epsilon}\in \Og_{X, \overline{x}},$ which proves the inclusion "$\subseteq$" from the Lemma. The other inclusion is obvious. The second claim from the Lemma now follows easily from the fact that $\Og_{X, \overline{x}}$ is a unique factorization domain. 
\end{proof}
\begin{proposition}
The morphism $g\colon X'\to Y$ is log smooth. \label{horizontalprop}
\end{proposition}
\begin{proof}
We shall break up the proof into several steps. Let $x\in X,$ and let $\overline{x}$ be a geometric point of $X$ above $x.$ Because of the local nature of our statement, we may assume that $Y$ is affine and that there is a global chart $Q\to \Gamma(Y, \Og_Y)$ of $\mathscr{M}_Y$ around $g(\overline{x}).$ Our other assumptions on $Y$ allow us also to assume that $Q$ is a finitely generated integral monoid. \\
\it Step 1: \rm \'etale-locally around $\overline{x}$, we shall construct a chart $P\to \mathscr{M}_{X, \overline{x}}$ of $\mathscr{M}_X$ such that $P$ is the direct sum of a finitely generated Abelian group and a finitely generated free monoid, together with an injective chart $Q \to P$ of $g.$ This step is an adaption of the proof of \cite[Theorem 3.5]{Kato}.
Locally around $\overline{x},$ we can choose global sections $t_1,..., t_r$ of $\mathscr{M}_X$ such that $ d\log t_1,..., d\log t_r$ are a basis of $\Omega^1_{X/Y}$ (we use this symbol to stand for log differentials). Exactly as in \it loc. cit., \rm we define a morphism 
$$\N^r\oplus Q\to \mathscr{M}_{X, \overline{x}},$$ given by $Q\to \mathscr{M}_{Y, g(\overline{x})}\to \mathscr{M}_{X, \overline{x}}$ and $e_i\to t_i$ for $i=1,...,r.$ This map gives rise to a homomorphism $\tilde{f}\colon \Z^r\oplus Q\gp\to \mathscr{M}^{\mathrm{gp}}_{X, \overline{x}}.$ The same argument as in \it loc. cit. \rm can be used to show that the cokernel of the induced map 
$f\colon \Z^r\oplus Q\gp \to \mathscr{M}^{\mathrm{gp}}_{X, \overline{x}}/\Og_{X, \overline{x}}^\times$ is finite and has cardinality invertible in $\Og_{X, \overline{x}}.$ By our assumption on $X,$ we know that $\mathscr{M}^{\mathrm{gp}}_{X, \overline{x}}/\Og_{X, \overline{x}}^\times$ is a free Abelian group of rank $m$ for some $m \in \N;$ hence the same is true for $\im f.$ \\
Let $s_1,..., s_m \in \mathscr{M}_{X, \overline{x}}^{\mathrm{gp}}$ be elements such that their images $[s_1],..., [s_m]\in \mathscr{M}^{\mathrm{gp}}_{X, \overline{x}}/\Og_{X, \overline{x}}^\times$ are a basis, and such that $\lambda_1[s_1],..., \lambda_m[s_m]$ are a basis of $\im f$ for suitable $\lambda_1,..., \lambda_m\in \N.$ Observe that all $\lambda_j$ are invertible in $\Og_{X, \overline{x}}.$ Further choose $\alpha_1,..., \alpha_m \in \Z^r\oplus Q\gp$ such that, for all $j=1,...,m,$ we have 
$$f(\alpha_j)=\lambda_j[s_j].$$ The choice of the $\alpha_j$ gives us a section $\im f\to \Z^r\oplus Q\gp,$ and hence also a retraction $\rho\colon \Z^r\oplus Q\gp\to \ker f.$ Combining these, we obtain a decomposition
$$\Z^r\oplus Q\gp=\ker f\oplus \im f.$$ We define 
$$G:=\ker f\oplus \mathscr{M}^{\mathrm{gp}}_{X, \overline{x}}/\Og_{X, \overline{x}}^\times.$$ Observe that $G$ is a finitely generated Abelian group. Our next aim will be to construct a morphism $h\colon G\to \mathscr{M}^{\mathrm{gp}}_{X, \overline{x}}$ extending $\tilde{f},$ such that the composition $G\to \mathscr{M}^{\mathrm{gp}}_{X, \overline{x}} \to  \mathscr{M}^{\mathrm{gp}}_{X, \overline{x}}/\Og_{X, \overline{x}}^\times$ is surjective. To this end, we shall need the following\\
\it Sublemma: \rm The map $[-]\colon \mathscr{M}^{\mathrm{gp}}_{X, \overline{x}}\to \mathscr{M}^{\mathrm{gp}}_{X, \overline{x}}/\Og_{X, \overline{x}}^\times$ admits a section $\kappa$ such that $\tilde{f}=\kappa\circ f$ on the image of the section $\im f\to \Z^r\oplus Q\gp.$ \\
To see this, observe that, for $j=1,..., m,$ we have
\begin{align*}
[\tilde{f}(\alpha_j)]&=f(\alpha_j)\\
&=\lambda_j[s_j]\\
&=[s_j^{\lambda_j}].
\end{align*}
This means that there exist $\epsilon_j\in \Og_{X, \overline{x}}^\times$ such that $\tilde{f}(\alpha_j)=\epsilon_j s_j^{\lambda_j}.$ As the $\lambda_j$ are invertible in $\Og_{X, \overline{x}},$ this means we can write
$$\tilde{f}(\alpha_j)=(\epsilon_j^{1/{\lambda_j}}s_j)^{\lambda_j}$$ for all $j=1,...,m.$ Hence $\kappa([s_j]):=\epsilon_j^{1/\lambda_j} s_j$ gives rise to a section with the desired properties.\\
We can now define $h$ by putting 
$$h((x,y)):=\tilde{f}(x)\kappa(y)$$ for $(x,y)\in G=\ker f\oplus \mathscr{M}^{\mathrm{gp}}_{X, \overline{x}}/\Og_{X, \overline{x}}^\times.$ An easy check shows that $h$ does indeed extend $\tilde{f}$. Furthermore, we observe that $G/(\Z^r\oplus Q\gp)\cong \mathrm{coker}\, f,$ which has cardinality invertible in $\Og_{X, \overline{x}}.$ As in \it loc. cit. \rm we see that the map $h^{-1}(\mathscr{M}_{X, \overline{x}})=:P\to \mathscr{M}_{X, \overline{x}}$ extends to a chart locally around $\overline{x}.$ From now on, we replace $X$ by such a neighbourhood of $\overline{x}.$ Similarly, we see that the map $Q\gp \to G=P\gp$ is injective, that its cokernel's torsion part is of cardinality invertible in $\Og_{X, \overline{x}},$ and that the map
$$X\to Y\times_{\Spec\Z[Q]}\Spec \Z[P]$$ is \'etale in the classical sense.\\
\it Step 2: \rm We shall now go on to constructing a chart $P'$ of $\mathscr{M}_X'$ locally around $\overline{x}$ which we can use to show that $g\colon X'\to Y$ is log smooth.
We can find elements $x_1,..., x_a\in \Og_{X, \overline{x}}$ ($a\in \N$) such that the sequence $x_1,..., x_a$ can be extended to a minimal sequence which generates the maximal ideal $\mathfrak{m}_{X,\overline{x}}\subseteq \Og_{X, \overline{x}},$ and such that $D$ is cut out by $x_1x_2...x_a,$ with $D_{\mathrm{vert}}$ being cut out by $x_1... x_b$ for some $0\leq b\leq a.$ Then $[x_1],..., [x_a]$ are a basis of $\mathscr{M}^{\mathrm{gp}}_{X, \overline{x}}/\Og_{X, \overline{x}}^\times,$ whereas $[x_1],..., [x_b]$ are a basis of $\mathscr{M}'^{\mathrm{gp}}_{X, \overline{x}}/\Og_{X, \overline{x}}^\times.$ Using Lemma \ref{stalklemma} as well as that $\Og_{X, \overline{x}}$ is regular, and hence a unique factorization domain, we see that the preimage of $\mathscr{M}_{X, \overline{x}}/\Og_{X, \overline{x}}^\times$ in $\mathscr{M}_{X, \overline{x}}^{\mathrm{gp}}$ is $\mathscr{M}_{X, \overline{x}}$ (and similarly for $\mathscr{M}'^{\mathrm{gp}}_{X, \overline{x}}$). Analogously, we see that the preimage of $\mathscr{M}'_{X, \overline{x}}/\Og_{X, \overline{x}}^\times$ in $\mathscr{M}_{X, \overline{x}}^{\mathrm{gp}}$ is $\mathscr{M}'_{X, \overline{x}}.$ 
Let $$G':= h^{-1}(\mathscr{M}_{X,\overline{x}}'^{\mathrm{gp}}).$$ Then the induced map $G'\to \mathscr{M}'^{\mathrm{gp}}_{X, \overline{x}}\to\mathscr{M}'^{\mathrm{gp}}_{X, \overline{x}}/\Og_{X, \overline{x}}^\times$ is surjective. If we now define $P'$ to be the preimage of $\mathscr{M}'_{X, \overline{x}}$ in $G'$ under $h,$ then \cite[Lemma 2.10]{Kato} tells us that the map $h\colon P'\to \mathscr{M}'_{X, \overline{x}}$ can be extended to a chart in an \'etale neighbourhood of $\overline{x}.$ Clearly, the map $Q\to P$ from the chart of $g$ constructed in Step 1 factors through $P'.$ We now replace $X$ by such a neighbourhood of $\overline{x}.$\\
\it Step 3: \rm We shall now see that the chart $Q\to P'$ of the morphism $g\colon X'\to Y$ (locally around $\overline{x}$) satisfies part (ii) of \cite[Theorem 3.5]{Kato} as relaxed by Remark 3.6 of \it loc. cit. \rm Indeed, we have
$P=\ker f\oplus \N\langle [x_1],..., [x_a]\rangle,$ and $P'=\ker f\oplus \N\langle [x_1],..., [x_b]\rangle.$ In particular, we have $P\cong P'\oplus \N^{a-b},$ and the morphism $Q\gp\to P\gp$ factors through the canonical inclusion $P'\gp\to P\gp=P'\gp\oplus \Z^{a-b}.$  Hence the torsion parts of the cokernels of $Q\gp\to P'\gp$ and $Q\gp\to P\gp$ are isomorphic, showing part (i) of \cite[ Theorem 3.5]{Kato} for the new log structure. To see part (ii), observe that we have a commutative diagram
$$\begin{tikzcd}
X \arrow{r} \arrow{rd} & Y\times_{\Spec\Z[Q]}\Spec \Z[P] \arrow{r}{\cong}\arrow{d}& (Y\times_{\Spec\Z[Q]}\Spec \Z[P'])\times_{\Spec\Z} \A^{a-b}_{\Z} \arrow{ld}{\pr_1}\\
&Y\times_{\Spec\Z[Q]}\Spec \Z[P']
\end{tikzcd}$$
showing that the map $X\to Y\times_{\Spec\Z[Q]}\Spec \Z[P']$ is smooth. This concludes the proof.
\end{proof}\\

\section{Logarithmic good reduction}
We shall now prove our main theorem, using the results established in the preceding sections. We begin with a technical lemma which will be needed later. 
\begin{lemma}
Let $K$ be a discretely valued field of residue characteristic not equal to 2 or 3, and let $E$ be an elliptic curve defined over $K.$ We shall denote the valuation ring of $K$ by $R,$ and the residue field (which we do not assume to be perfect) by $k.$ Then the finite \'etale $K$-group scheme $E[3]$ is tamely ramified with respect to $R.$ \label{3torslem}
\end{lemma}
\begin{proof}
We may replace $R$ by its completion and assume that $R$ is complete. We can find finite separable extensions $E_1,..., E_n$ such that there is an isomorphism
$$E[3]\cong \Spec (K\times E_1\times...\times E_n)$$ as schemes, and we know that
$$1+\sum_{j=1}^n [E_j:K]=\ord_K E[3]=9.$$ Now observe that the $K$-scheme $E[3]$ admits the non-trivial involution $\phi:=[-1].$ Suppose $x$ is a (set-theoretic) point of $E[3]$ with residue field $E_j$ for some $j=1,...,n.$ If $\phi(x)\not=x$, then the residue field at $\phi(x)$ is $E_i$ for some $i\not=j.$ In particular, we obtain an isomorphism $E_j\to E_i$ of extensions of $K,$ which implies $[E_j:K]\leq 4.$ In particular, the degree of the induced extension of residue fields must divide 3 or 4, as must the ramification index. This implies that $E_j$ is a tamely ramified extension of $K.$ On the other hand, suppose that $\phi(x)=x.$ Then $\phi^\ast$ induces a non-trivial $K$-automorphism of $E_j$ of order 2. Indeed, if the automorphism of $E_j$ induced by $\phi$ were trivial, the $K\sep$-points of $E_j$ would be $K\sep$-points of $E[3]\backslash \{0\}$ fixed by $[-1],$ which is impossible since $E[3](K\sep)\cong \Z/3\Z\times \Z/3\Z.$ If $E_j^\phi$ denotes the field of elements of $E_j$ fixed by $\phi^\ast,$ then $E_j^\phi\subseteq E_j$ is a Galois extension of degree 2, so $[E_j:K]$ must be even. Since the only prime numbers which divide even positive integers not exceeding 8 are 2 and 3, we conclude as in the previous case.
\end{proof}\\
We are now ready to prove the main theorem. Recall that $K$ is a complete discretely valued field with residue field $k,$ which we assume to be algebraically closed and of characteristic different from 2 and 3. As before, we choose an algebraic closure $K\alg$ of $K$ and let $K\sep$ be the separable closure of $K$ in $K\alg.$ Further recall that $f\colon X\to C$ is a minimal elliptic surface over $K$ which admits a section $\sigma\colon C\to X.$  We let $D$ be the formal sum of closed points of $C$ above which the fibre of $f$ is singular. Then $D$ is a divisor on $C$ which is \'etale and tamely ramified over $K$ (see Theorem \ref{Newtamethm} and Proposition \ref{separableprop}), provided that Condition 1 or Condition 2 are satisfied. 
\begin{theorem}
Assume that the minimal elliptic surface $X$ satisfies Condition 1 or Condition 2 above. Suppose further that there is a divisor $A$ on $C$ which is \'etale and tamely ramified over $K$ such that $$\deg_K D+\deg_KA+ 2\dim_K H^1(C, \Og_C) > 2,$$ and such that the supports of $D$ and $A$ do not intersect. Then there exists a regular scheme $\mathscr{X}$ and a projective morphism $\mathscr{X}\to \Spec \Og_K$ with smooth generic fibre, such that\\
(i) the morphism $\mathscr{X}\to \Spec \Og_K$ is log smooth with respect to the log structures defined by the special fibres, and \\
(ii) there exists a projective modification $\mu\colon\mathscr{X}_K\to X.$\\
Moreover, the morphism $\mathscr{X}\to \Spec \Og_K$ factors as $\mathscr{X}\overset{\phi}{\to}\mathscr{C}\to \Spec \Og_K$, where $\mathscr{C}$ is a regular and projective model of $C$, log smooth over $\Og_K$ with respect to the special fibres, such that the diagram \label{mainthm}
$$\begin{tikzcd}
\mathscr{X}_K \arrow{r}{\mu} \arrow{rd}[swap]{\phi_K} & X \arrow{d}{f} \\
&\mathscr{C}_K =&\!\!\!\!\!\!\!\!\!\!\!\!\!\!\!\!\!\!\!\!\!\!\!\!\! C
\end{tikzcd}$$
commutes. The exceptional locus of $\mu$ is mapped into the support of $D$ by $f.$ 
\end{theorem}
\begin{proof}
We shall begin by constructing a projective and regular model $\mathscr{C}\to \Spec \Og_K$ of $C,$ flat over $\Og_K,$ with the following property: Let $U_{\mathscr{C}}$ be the complement of the subscheme $D\subseteq C.$ Then the reduced complement of $U_{\mathscr{C}}$ in $\mathscr{C}$ is a divisor with normal crossings, and the morphism $\mathscr{C}\to \Spec \Og_K$ is log smooth with respect to the log structure on $\mathscr{C}$ induced by $U_{\mathscr{C}}$ 
(unless mentioned otherwise, the log structure on $\Spec \Og_K$ will always be the one induced by the special point). Let $n>1$ and let $U_{\mathscr{C}}'$ be the complement of $A+D$ in $C$, viewed as an open subscheme of $\mathscr{C}.$ If Condition 1 holds, then because $H^1(C_{K\alg}, \Z_\ell)\subseteq H^1(C_{K\alg}, \Q_\ell)\subseteq H^1(X'_{K\alg}, \Q_\ell),$ we see that $H^1(C_{K\alg}, \Z_\ell)$ is tamely ramified. If Condition 2 holds, the same conclusion follows immediately. 
This implies that $H^1(C_{K\alg}, \Z/\ell^n\Z)\cong \Pic^0_{C_{K\alg}/K\alg}[\ell^n](K\alg)$ is tamely ramified as well. Furthermore, $A+D$ is \'etale and tamely ramified over $K$ (Theorem \ref{Newtamethm}). Hence \cite[Theorem 1]{Saito}, implication $(3)\Rightarrow (1)$ (with $N=\ell ^n$) tells us that there exists a model $\mathscr{C}$ of $C$ which is projective, regular, and flat over $\Og_K$ which is furthermore log smooth with respect to the log structure induced by $U'_{\mathscr{C}}.$ However, Proposition \ref{horizontalprop} now tells us that we my replace $U_{\mathscr{C}}'$ by $U_{\mathscr{C}}$ without affecting the conclusion.\\
The map $f$ induces an elliptic curve $X_{U_{\mathscr{C}}}\to U_{\mathscr{C}},$ and the section $\sigma$ of $f$ induces a divisor on $X_{U_{\mathscr{C}}}$ which is obviously \'etale, tamely ramified, and of degree 1 over $K.$ Hence \cite[Theorem 1]{Saito}, implication $(3)\Rightarrow (1)$ (with $N=3$) together with Lemma \ref{3torslem} tell us that we can extend $X_{U_{\mathscr{C}}}\to U_{\mathscr{C}}$ to a regular, projective and flat scheme $\mathscr{X}\to \mathscr{C},$ such that, with respect to the log structure on $\mathscr{X}$ induced by the complement of the section $\sigma$ on $X_{U_{\mathscr{C}}},$ the morphism $\mathscr{X}\to\mathscr{C}$ is log smooth. Clearly, $\mathscr{C}_K\to C$ is a proper regular model of the generic fibre of $f,$ so there exists a unique modification $\mu\colon \mathscr{X}_K\to X$ compatible with the morphisms down to $C.$ It is also clear that the map $\mathscr{X}\to \Spec \Og_K$ is log smooth with respect to the log structure on $\mathscr{X}$ induced by $U_{\mathscr{X}}.$ By Proposition \ref{horizontalprop} we may replace $U_{\mathscr{X}}$ by the generic fibre $\mathscr{X}_K$ of $\mathscr{X}\to \Spec \Og_K$ without affecting the logarithmic smoothness of the morphism, thus concluding the proof. 
\end{proof}\\
In some situations, the conclusions of the last Theorem can be strengthened further:
\begin{theorem} 
Let $X\to C$ be an elliptic surface such that $f$ is smooth in the classical sense (which implies in particular that the elliptic surface is minimal). Suppose the Galois representation on $H^1(X_{K\alg}, \F_\ell)$ is tamely ramified and that there exists an effective divisor $A$ on $C$ which is \'etale and tamely ramified and such that 
 $$\deg A+2\dim_K H^1(C, \Og_C)> 2.$$ Then the statement of Theorem \ref{mainthm} can be modified to include the requirement that $\mu$ be an isomorphism. In particular, $X$ has logarithmic good reduction. \label{smooththm}
\end{theorem}
\begin{proof}
Our assumptions imply that for $n>1,$ the Galois representation on $$H^1(C_{K\alg}, \Z/\ell^n\Z)\cong \Pic^0_{C_{K\alg}/K\alg}[\ell^n](K\alg)$$ is tamely ramified (note that the condition from the Theorem follows from Condition 1). As in the proof of Theorem \ref{mainthm} we deduce that there exists a proper, flat, and regular model $\mathscr{C}\to \Spec\Og_K$ which is log smooth with respect to the log structure given by the special fibre. Using \cite[Theorem 1]{Saito}, implication $(3) \Rightarrow (1)$ again, we see that $X\to C$ can be extended to a log smooth and projective map $\mathscr{X}\to \mathscr{C},$ and the resulting morphism $\mathscr{X}\to \Spec \Og_K$ is clearly log smooth. Now we use Proposition \ref{horizontalprop} as above. 
\end{proof}

\end{document}